\documentclass[12pt]{amsart}
\usepackage{amsmath,amssymb,amsthm,amsfonts,url,enumerate,geometry}
\author{Greg Martin}
\author{Paul Pollack}
\address{University of British Columbia\\ Department of Mathematics \\ Room 121\\ 1984 Mathematics Road\\Vancouver, BC Canada V6T 1Z2}
\email{gerg@math.ubc.ca} 
\email{pollack@math.ubc.ca}
\title[The average least character nonresidue]{The average least character nonresidue and further variations on a theme of Erd\H{o}s}
\DeclareMathAlphabet{\curly}{U}{rsfs}{m}{n}
\newtheorem{thm}{Theorem}[section]
\newtheorem{prop}[thm]{Proposition}
\newtheorem{cor}[thm]{Corollary}

\newtheorem{lem}[thm]{Lemma}
\theoremstyle{definition}
\newtheorem*{rmk}{Remark}

\numberwithin{equation}{section}
\begin{document}
\renewcommand{\labelenumi}{(\roman{enumi})}
\def\e{e}
\def\A{\curly{A}}
\def\B{\curly{B}}
\def\C{\mathbb{C}}
\def\N{\mathbb{N}}
\def\Q{\mathbb{Q}}
\def\Z{\mathbb{Z}}
\def\R{\mathbb{R}}
\def\Pp{\curly{P}}
\def\X{\curly{X}}
\def\Li{\mathrm{Li}}
\newcommand{\leg}[2]{\genfrac{(}{)}{}{}{#1}{#2}}
\renewcommand{\pmod}[1]{{\ifmmode\text{\rm\ (mod~$#1$)}\else\discretionary{}{}{\hbox{ }}\rm(mod~$#1$)\fi}}


\begin{abstract}  For each nonprincipal Dirichlet character $\chi$, let $n_{\chi}$ be the least $n$ with $\chi(n)\not\in \{0,1\}$. For each prime $p$, let $\chi_p(\cdot)=\leg{\cdot}{p}$ be the quadratic character modulo $p$. In 1961, Erd\H{o}s showed that $n_{\chi_p}$ possesses a finite mean value as $p$ runs over the odd primes in increasing order. We show that as $q\to\infty$, the average of $n_{\chi}$ over \emph{all} nonprincipal characters $\chi$ modulo $q$ is $\ell(q) +o(1)$, where $\ell(q)$ denotes the least prime not dividing $q$. Moreover, if one averages over all nonprincipal characters of moduli $\leq x$, the average approaches (as $x\to\infty$) the limiting value
\[ \sum_{\ell} \frac{\ell^2}{\prod_{p \leq \ell}(p+1)} \approx 2.5350541804,\]
where the sum is over primes $\ell$ and the product is over primes $p \leq \ell$.
	
One can also view Erd\H{o}s's  theorem as giving the average size of the least non-split prime in the quadratic field of conductor $p$, where again $p$ runs over odd primes. Similar results with the average taken over all quadratic fields were recently proved by the second author. In this paper, we prove a result of this type for cubic number fields: If one averages over all cubic fields $K$, ordered by discriminant, then the mean value of the least rational prime that does not split completely in $K$ is
	\[ \sum_{r} \frac{r(5/6+1/r+1/r^2)}{1+1/r+1/r^2} \prod_{p < r}\frac{1/6}{1+1/p+1/p^2} \approx 2.1211027269,\]
where the sum is over all primes $r$.
\end{abstract}
\maketitle

\section{Introduction}
For a nonprincipal Dirichlet character $\chi$, let $n_{\chi}$ denote the least character nonresidue for $\chi$, that is, the least positive integer $n$ with $\chi(n) \ne0$ and $\chi(n)\ne1$. Under the assumption of the generalized Riemann hypothesis for Dirichlet $L$-functions, we know that $n_{\chi} \leq 3\log^2q$ for every nonprincipal character $\chi\pmod q$. (As stated this result is due to Bach~\cite[Theorem 3]{bach90}, although the first result of its kind was proved by Ankeny~\cite{ankeny52}. If $q$ is prime, then $\chi$ is a $k$th power residue character for some $k$ dividing $q-1$, and the study of the maximal order of $n_{\chi}$ goes back to Vinogradov and Linnik in the early part of the twentieth century.) The best unconditional result in this direction, due to Norton (see \cite[equation (1.22)]{norton98}), asserts that $n_\chi \ll_{\epsilon} q^{{1}/{4\sqrt{\e}}+\epsilon}$ for every nonprincipal character $\chi\pmod q$.

Short of a satisfactory unconditional ``pointwise'' result, one can study $n_{\chi}$ on average. Erd\H{o}s~\cite{erdos61} was the first to adopt this viewpoint in his treatment of quadratic characters $\chi(\cdot) = \leg{\cdot}{p}$ modulo primes~$p$. He showed that
\begin{equation}\label{eq:erdos0}
\lim_{x\to\infty} \bigg( \pi(x)^{-1} \sum_{2 < p \leq x} n_{\leg{\cdot}{p}} \bigg) = \sum_{k=1}^{\infty}\frac{p_k}{2^k}, \end{equation}
where $p_k$ denotes the $k$th prime in increasing order. In other words, the average value of the least quadratic nonresidue modulo a prime is the constant $\sum_{k=1}^{\infty} {p_k}/{2^k} \approx 3.67464$.
This result was extended to all real primitive characters by the second author \cite{pollack11}, who showed that
\begin{equation}\label{eq:erdos1}
\lim_{x\to\infty} \bigg( {\sum_{|D| \leq x}1} \bigg)^{-1} \bigg( \sum_{|D| \leq x} n_{\leg{D}{\cdot}} \bigg) = \sum_\ell \frac{\ell^2}{2(\ell+1)} \prod_{p<\ell} \frac{p+2}{2(p+1)},
\end{equation}
where the sums on the left-hand side are over fundamental discriminants $D$ whose absolute value is at most $x$, and the sum and product on the right-hand side are over primes $\ell$ and~$p$. In other words, the average least character nonresidue of a general quadratic character is the constant on the right-hand side of equation~\eqref{eq:erdos1}, which is approximately $4.98085$.

Our first goal in this paper is to understand the average of $n_{\chi}$ taken over {\em all} nonprincipal characters $\chi$. Let $\ell(q)$ denote the least prime not dividing $q$. If $\chi$ is any character modulo $q$, then $\chi(n)=0$ whenever $1 < n < \ell(q)$; hence $n_{\chi} \geq \ell(q)$ for all nonprincipal characters~$\chi$. We prove that the average of $n_{\chi}$ over all nonprincipal characters modulo $q$ is always very close to $\ell(q)$:

\begin{thm}\label{thm:main} For $q\geq 3$, we have
\[ \frac{1}{\phi(q)-1}\sum_{\substack{\chi\pmod q \\ \chi \neq \chi_0}} n_{\chi} = \ell(q) + O\bigg( \frac{(\log\log{q})^3}{\log{q}} \bigg). \]
\end{thm}

One consequence of Theorem \ref{thm:main} is an analogue of the average value results \eqref{eq:erdos0} and \eqref{eq:erdos1}, but with the average taken over all nonprincipal characters with conductor at most~$x$:

\begin{cor}\label{cor:finalavg} Define
\begin{equation}\label{eq:Delta def}
\Delta=\sum_{\ell} \frac{\ell^2}{\prod_{p \leq \ell}(p+1)} \approx 2.5350541804,
\end{equation}
where the sum and product are taken over primes $\ell$ and~$p$. Then
\[
\lim_{x\to\infty} \bigg( \sum_{q \leq x}\sum_{\substack{\chi \pmod{q}\\\chi\neq\chi_0}}1 \bigg)^{-1} \bigg( \sum_{q \leq x}\sum_{\substack{\chi \pmod{q}\\\chi\neq\chi_0}} n_{\chi} \bigg) = \Delta.
\]
\end{cor}

We remark that our methods can also address the analogues of Theorem~\ref{thm:main} and Corollary~\ref{cor:finalavg} in which the sums over characters $\chi$ are restricted to primitive characters only (see equation~\eqref{ell avg for primitive} and Theorem~\ref{finalavg primitive}).


One can also view \eqref{eq:erdos0} and \eqref{eq:erdos1} as results in algebraic number theory. If $\chi$ is the quadratic character modulo $p$, then $n_{\chi}$ is the least prime that does not split in the quadratic field of conductor $p$ (equivalently, of discriminant $(-1)^{(p-1)/2}p$). Thus, Erd\H{o}s's theorem gives the average least non-split prime in a quadratic extension, where the average is restricted to quadratic fields of prime conductor. Similarly,  \eqref{eq:erdos1} is an estimate for the average least inert prime, where now the average is taken over all quadratic fields, ordered by discriminant.

For a prime $p\equiv 1\pmod{3}$, define $n_3(p)$ as the least cubic nonresidue modulo $p$. Elliott \cite{elliott68} showed that $n_3(p)$ possesses a finite mean-value. (In fact, he showed the analogous result for $k$th power nonresidues, for all $k$.) We can interpret Elliott's result on $n_3(p)$ as the determination of the average least non-split prime, with the average restricted to cyclic cubic extensions of prime conductor. Our second theorem gives the average least non-split prime, with the average taken over \emph{all} cubic extensions of $\Q$, ordered by discriminant. In what follows, we write $D_K$ for the discriminant of the number field $K$.

\begin{thm}\label{thm:cubic} For a cubic number field $K$, let $n_K$ denote the least rational prime that does not split completely in $K$. Define
\begin{equation} \label{eq:Delta_nsc def}
\Delta_{nsc} = \sum_\ell \frac{5\ell^3+6\ell^2+6\ell}{6(\ell^2+\ell+1)} \prod_{p < \ell}\frac{p^2}{6(p^2+p+1)}\approx 2.1211027269,
\end{equation}
where the sum and product are taken over primes $\ell$ and~$p$. Then
\[
\lim_{x\to\infty} \bigg( \sum_{|D_K| \leq x} 1 \bigg)^{-1} \bigg( \sum_{|D_K| \leq x} n_K \bigg) = \Delta_{nsc},
\]
where the sums on the left-hand side are taken over (all isomorphism classes of) cubic fields $K$ for which $|D_K| \leq x$.
\end{thm}

\begin{rmk} As a reality check, we sampled cubic fields with discriminant near $-10^{12}$. Using K.\ Belabas's \texttt{cubics} package (see \cite{belabas97, belabas04}) and \texttt{PARI/GP}, we found that the average of $n_K$ over cubic fields $K$ with $|D_K+10^{12}| \leq 250$,$000$ is about $2.12065$, a close match to the limiting constant $\Delta_{nsc}$.
\end{rmk}

The subscript of $\Delta_{nsc}$, which stands for ``not split completely'', suggests that one can also calculate the average value of the least prime with other behaviors (split completely, or partially split, or inert) in cubic extensions of $\Q$. We can indeed establish these additional average values (see Theorem~\ref{thm:generalsplit}), but only under the assumption of a generalized Riemann hypothesis.

The proofs of our theorems, while similar in flavor to the arguments of~\cite{erdos61, pollack11}, employ different tools. For the proof of Theorem~\ref{thm:main}, our primary inspiration was a paper of Burthe~\cite{burthe97}, which uses zero-density estimates and a theorem of Montgomery (Proposition~\ref{prop:montgomery} below) to prove that
\[ \frac{1}{x}\sum_{q \leq x} \max_{\substack{\chi \pmod{q} \\ \chi \neq \chi_0}} n_{\chi} \ll (\log{x})^{97}; \]
note that Burthe's result shows unconditionally that a bound of the same flavor as Bach's $n_\chi \le 3\log^2q$ holds on average. The proof of Theorem~\ref{thm:cubic} involves a few different ingredients; perhaps most crucial is the recent work of Taniguchi and Thorne~\cite{TT11} on counting cubic fields with prescribed local conditions (see Proposition~\ref{prop:TT}).

\subsection*{Notation} The letters $\ell$, $p$, and $r$ are reserved for prime variables. We write $P(n)$ for the largest prime factor of $n$ and $\ell(n)$ for the smallest prime {\em not} dividing~$n$. We say that $n$ is \emph{$y$-friable} (or \emph{$y$-smooth}) if $P(n)\leq y$, and we let $\Psi(x,y)$ denote the number of $y$-friable integers not exceeding~$x$. We write $\omega(n)=\sum_{p \mid n}1$ for the number of distinct prime factors of $n$ and $\Omega(n)=\sum_{p^k \mid n}1$ for the number of prime factors of $n$ counted with multiplicity. We use $c_1, c_2, \dots$ for absolute positive constants.

\section{The average least character nonresidue\pmod q} \label{main thm section}

We begin by quoting two theorems from the literature. The first, due to Montgomery~\cite[Theorem 1, p.\ 164]{montgomery94} (compare with~\cite{LMO79}), relates the size of $n_{\chi}$ to a zero-free region for $L(s,\chi)$ near $s=1$. Define $N(\sigma, T, \chi)$ as the number of zeros $s=\beta+i\gamma$ of $L(s,\chi)$ with $\sigma \leq \beta \leq 1$ and $|\gamma| \leq T$.

\begin{prop}\label{prop:montgomery} Let $\chi$ be a nonprincipal Dirichlet character modulo $q$, and let $\delta$ be a real number in the range $1/\log q  < \delta \leq \frac12$. If $N(1-\delta, \delta^2 \log{q}, \chi)=0$, then $n_{\chi} < (c_1 \delta \log{q})^{1/\delta}$. Here $c_1>0$ is an absolute constant.
\end{prop}

The second, due to Jutila~\cite[Theorem 1]{jutila77}), is a zero-density estimate for the collection of characters to a given modulus.

\begin{prop}\label{prop:jutila} Let $\epsilon > 0$. For $4/5 \leq \alpha \leq 1$ and $T \geq 1$, we have
\[ \sum_{\chi \pmod{q}} N(\alpha, T, \chi) \ll_{\epsilon} (qT)^{(2+\epsilon)(1-\alpha)}. \] \end{prop}

These propositions allow us to establish the next lemma, which will eventually be used to show that characters $\chi$ with $n_{\chi}$ larger than about $(\log{q})^5$ do not significantly affect the average of $n_{\chi}$.

\begin{lem}\label{lem:nottoomany} There exists an absolute constant $c_2>0$ such that for any positive integer $q$, the number of characters $\chi\pmod q$ with $n_{\chi} \ge c_2 \log^5q$ is $\ll q^{0.43}$.
\end{lem}

\begin{proof}
We may assume that $q\ge e^{25}$, and we set $c_2 = (\frac{c_1}5)^5$ where $c_1$ is the constant from Proposition~\ref{prop:montgomery}. That proposition with $\delta = \frac15$ implies that the number of nonprincipal $\chi$ with $n_{\chi} \geq (\frac{c_1}{5}\log{q})^5 = c_2\log^5q$ is bounded above by
\[
\#\bigg\{ \chi \pmod q \colon N\bigg(\frac{4}{5}, \frac{\log q}{25}, \chi\bigg) \ge 1 \bigg\} \le \sum_{\chi \pmod{q}} N\bigg(\frac{4}{5}, \frac{\log q}{25}, \chi\bigg).
\]
From Proposition~\ref{prop:jutila} with $\alpha=\frac45$ and $\epsilon = \frac{1}{10}$, this sum is $\ll (q\log{q})^{21/50} \ll q^{0.43}$.
\end{proof}

We also need an elementary lemma concerning the structure of the character group.

\begin{lem}\label{lem:vanishing characters}
Let $G$ be a finite abelian group and $H\le G$ a subgroup. The number of homomorphisms $\chi \colon G\to\C^{\times}$ such that $\chi(h)=1$ for all
 $h\in H$ is exactly $\#G/\#H$. In particular, if $H$ is a subgroup of the multiplicative group $(\Z/q\Z)^\times$, then the number of Dirichlet characters $\chi\pmod q$ such that $\chi(h)=1$ for all $h\in H$ is exactly $\phi(q)/\#H$.
\end{lem}

The proof of Lemma \ref{lem:vanishing characters} is simple; the homomorphisms $\chi\colon G\to\C^{\times}$ that are trivial on $H$ are in canonical bijection with  the homomorphisms $\tilde{\chi}\colon G/H \to \C^{\times}$, the number of which is $\#G/\#H$ (compare with \cite[Chapter VI, \S1.1]{serre73}).


The next lemma reduces the proof of Theorem~\ref{thm:main} to the task of bounding two expressions in its error term. In addition to the notation $\ell(q)$ for the least prime not dividing $q$, we define $f(q)$ to be the multiplicative order of $\ell(q)$ modulo~$q$. We record the helpful estimate $\ell(q) \ll \log q$, which holds because the product of all the primes less than $2\log q$ (say) is larger than $q$ when $q$ is sufficiently large, by the prime number theorem.

\begin{lem}\label{lem:rewrite}
Let $\chi$ be a nonprincipal character$\pmod q$. For any real number $Y$ sat\-isfying $\ell(q)<Y\le c_2\log^5 q$ (where $c_2$ is the constant from Lemma~\ref{lem:nottoomany}), we have
\[
\frac{1}{\phi(q)-1} \sum_{\substack{\chi\pmod q \\ \chi \ne \chi_0}} n_{\chi} = \ell(q) + O\bigg( \frac1{\phi(q)} \sum_{\substack{\chi\pmod q \\ \chi \ne \chi_0 \\ Y < n_\chi \le c_2\log^5q}} n_\chi + \frac Y{f(q)} + q^{-1/3} \bigg).
\]
\end{lem}

\begin{proof}
We begin by writing
\begin{multline}
\label{four-way split}
\frac{1}{\phi(q)-1} \sum_{\substack{\chi\pmod q \\ \chi \ne \chi_0}} n_{\chi} = \frac{1}{\phi(q)-1} \sum_{\substack{\chi\pmod q \\ \chi \ne \chi_0 \\ n_\chi=\ell(q)}} n_{\chi} \\
+ O\bigg( \frac{1}{\phi(q)} \sum_{\substack{\chi\pmod q \\ \chi \ne \chi_0 \\ \ell(q) < n_\chi \le Y}} n_{\chi} + \frac{1}{\phi(q)} \sum_{\substack{\chi\pmod q \\ \chi \ne \chi_0 \\ Y < n_\chi \le c_2\log^5q}} n_{\chi} + \frac{1}{\phi(q)} \sum_{\substack{\chi\pmod q \\ \chi \ne \chi_0 \\ n_\chi > c_2\log^5q}} n_{\chi} \bigg).
\end{multline}
By Lemma~\ref{lem:vanishing characters} applied to the subgroup of $(\Z/q\Z)^\times$ generated by $\ell(q)$, which has order $f(q)$, the number of characters $\chi \pmod q$ with $\chi(\ell(q))=1$ equals $\phi(q)/f(q)$. It follows that 
\begin{align*}
\frac{1}{\phi(q)-1} \sum_{\substack{\chi\pmod q \\ \chi \ne \chi_0 \\ n_\chi=\ell(q)}} n_{\chi} &= \frac{\phi(q)-\phi(q)/f(q)}{\phi(q)-1} \ell(q) \\
&= \ell(q) + O\bigg( \frac{\ell(q)}{f(q)} \bigg) = \ell(q) + O\bigg( \frac Y{f(q)} \bigg),
\end{align*}
while the first error term in equation~\eqref{four-way split} can be bounded by
\[
\frac{1}{\phi(q)} \sum_{\substack{\chi\pmod q \\ \chi \ne \chi_0 \\ \ell(q) < n_\chi \le Y}} n_{\chi} \le \frac{1}{\phi(q)} \bigg( \frac{\phi(q)}{f(q)} - 1 \bigg) Y \ll \frac Y{f(q)}.
\]
As for the last error term in equation~\eqref{four-way split}, we use the fact that $n_\chi \ll q^{1/5}$ (this follows from Norton's result quoted in the introduction, since $\frac1{4\sqrt e} < \frac15$). Using this estimate in conjunction with Lemma~\ref{lem:nottoomany}, we have the bound
\[
\frac{1}{\phi(q)} \sum_{\substack{\chi\pmod q \\ \chi \ne \chi_0 \\ n_\chi > c_2\log^5q}} n_{\chi} \ll \frac1{\phi(q)} q^{0.43} \cdot q^{1/5} \ll q^{-1/3}
\]
(here we have used $\phi(q) \gg q/\log\log q \gg q^{0.99}$). The lemma now follows from equation~\eqref{four-way split} and the subsequent estimates.
\end{proof}

Before we proceed, we must quote one more theorem from the literature. This result, due to Baker and Harman~\cite{BH96, BH98}, asserts that many shifted primes possess a large prime factor.

\begin{prop}\label{prop:BH} For each positive real number $\theta \leq 0.677$, there is a constant $c_{\theta} > 0$ with the following property: For $x$ sufficiently large in terms of $\theta$, the number of primes $p \leq x$ with $P(p-1) > x^{\theta}$ is $> c_{\theta}x/\log{x}$.
\end{prop}

The following somewhat strange dichotomy will be useful to us in the proof of Theorem~\ref{thm:main}. For the rest of this section, define
\[
X = X(q) = \frac{(\log\log q)^3}{\log\log\log q}.
\]

\begin{lem}
\label{lem:strange dichotomy}
If $q$ is a sufficiently large integer, then either $f(q) \gg \exp((\log\log q)^2)$ or there exist at least six primes less than $X$ that do not divide $q$.
\end{lem}

\begin{rmk}
The proof actually yields many more than six small primes: when $f(q)$ is small, we could conclude that there are $\gg X/\log X$ primes less than $X$ that do not divide~$q$. We will not need this stronger conclusion, however.
\end{rmk}

\begin{proof}
Define $S_0$ to be the set of primes $p\le X$ dividing $q$ and satisfying $P(p-1) > X^{0.67}$. By Proposition~\ref{prop:BH}, there are $\gg X/\log{X}$ primes $p \leq X$ with $P(p-1) > X^{0.67}$. Let us assume that there are at most five primes less than $X$ that do not divide $q$ (so that we want to derive the lower bound $f(q) \gg \exp((\log\log q)^2)$); then almost all of these primes obtained from Proposition~\ref{prop:BH} are actually divisors of $q$ (when $q$ is sufficiently large), and so $\#S_0 \gg X/\log X$.

Define $S$ to be the subset of $S_0$ consisting of those primes in $S_0$ having the additional property that the multiplicative order of $\ell(q)$ modulo $p$ is divisible by $P(p-1)$. We claim that $\#S\gg X/\log X$ as well. To see this, note that if $p$ does not have this property, then the order of $\ell(q)$ modulo $p$ is a divisor of $(p-1)/P(p-1)$ and so is less than $X^{0.33}$; hence, $p$ divides an integer of the form $\ell(q)^j-1$ for some positive integer $j\le X^{0.33}$. On the other hand, each number $\ell(q)^j-1$ has at most $\log(\ell(q)^j-1)/\log2 \ll j\log \ell(q)$ prime factors; furthermore, $\ell(q) \ll \log q$. Therefore the total number of prime factors of all the numbers $\ell(q)^j-1$ is
\[
\ll \sum_{j\le X^{0.33}} j\log \ell(q) \ll X^{0.66} \log\log q \ll X^{0.995}
\]
by the definition of~$X$. Therefore $\#(S_0\setminus S) \ll X^{0.995} = o(X/\log X)$, and hence $\#S \gg X/\log X$ as claimed.

If $r$ is a prime exceeding $X^{0.67}$, then the number of $p \in S$ for which $P(p-1)=r$ is clearly bounded by the number of integers in $(1,X]$ that 
are congruent to $1\pmod r$, which is at most $X/r < X^{0.33}$. Hence, the number of distinct values of $P(p-1)$, as $p$ ranges over $S$, is $\gg X^{0.67}/\log{X}$. Since each prime $p$ divides $q$, we know that $f(q)$ is divisible by the order of $\ell(q)\pmod p$ for all $p\in S$, hence divisible by all of these $\gg X^{0.67}/\log{X}$ distinct primes $P(p-1)$, each of which exceeds $X^{0.67}$. It follows that there exist positive constants $c_3$ and $c_4$ such that
\begin{align*}
f(q) \ge (X^{0.67})^{c_3 X^{0.67}/\log{X}} &\ge \exp(c_4 X^{0.67}) \\
&= \exp \bigg( \frac{c_4(\log\log q)^{2.01}}{(\log\log\log q)^{0.67}} \bigg) \gg \exp\big( (\log\log q)^2\big)
\end{align*}
as desired.
\end{proof}

\begin{proof}[Proof of Theorem~\ref{thm:main}]
We may assume that $q$ is sufficiently large. Suppose first that $f(q) \gg \exp((\log\log q)^2)$. Taking $Y=c_2\log^5 q$ in Lemma~\ref{lem:rewrite}, we see that
\[
\frac{1}{\phi(q)-1} \sum_{\substack{\chi\pmod q \\ \chi \ne \chi_0}} n_{\chi} = \ell(q) + O\bigg( \frac{\log^5q}{f(q)} + q^{-1/3} \bigg) = \ell(q) + O\bigg( \frac1{\exp((\log\log q)^{3/2})} \bigg),
\]
say, which is stronger than the error term in the statement of the theorem. By Lemma~\ref{lem:strange dichotomy}, the only case left to consider is the case where there exist at least six primes $p_1,\dots,p_6$ less than $X$ that do not divide~$q$.

In particular, $\ell(q) < X$ in this case, so we may take $Y=X$ in Lemma~\ref{lem:strange dichotomy} to obtain
\begin{align*}
\frac{1}{\phi(q)-1} \sum_{\substack{\chi\pmod q \\ \chi \ne \chi_0}} n_{\chi} &= \ell(q) + O\bigg( \frac1{\phi(q)}\sum_{\substack{\chi\pmod q \\ \chi \ne \chi_0 \\ X < n_\chi \le c_2\log^5q}} n_\chi + \frac X{f(q)} + q^{-1/3} \bigg) \\
&= \ell(q) + O\bigg( \frac{\log^5q}{\phi(q)} \#\big\{ \chi \pmod q \colon n_\chi > X \big\} + \frac X{f(q)} + q^{-1/3} \bigg).
\end{align*}
Since $q \mid (\ell(q)^{f(q)}-1)$, we have the trivial lower bound $f(q) \ge (\log q)/\log\ell(q)$, and so $X/f(q) \le (X \log\ell(q))/\log{q} < X\log{X}/\log{q}\ll (\log\log{q})^3/\log{q}$. The theorem therefore follows if we can show that
\begin{equation}
\label{not too many characters}
\#\big\{ \chi \pmod q \colon n_\chi > X \big\} \ll \phi(q) \frac{(\log\log q)^3}{(\log q)^6}.
\end{equation}

Let $H$ be the subgroup of $(\Z/q\Z)^{\times}$ generated by (the images of) $p_1, \dots, p_6$. If $n_{\chi}>X$, then $\chi(h)=1$ for all $h\in H$; thus by Lemma~\ref{lem:vanishing characters},
\[
\#\big\{ \chi \pmod q \colon n_\chi > X \big\} \le \#\big\{ \chi \pmod q\colon \chi(h)=1 \text{ for all }h\in H \big\} = \frac{\phi(q)}{\#H}.
\]
However, the order of $H$ is at least the number of integers less than $q$ that factor as a product of the $p_i$, which includes every integer of the form $p_1^{\alpha_1} \cdots p_6^{\alpha_6}$ with 
$0\le \alpha_1,\dots,\alpha_6 < \log{q}/(6\log{X})$. Therefore $\#H \ge 6^{-6} ((\log{q})/\log{X})^6$, and so
\[
\#\big\{ \chi \pmod q \colon n_\chi > X \big\} \ll \frac{\phi(q)}{((\log q)/\log X)^6} \ll \phi(q) \frac{(\log\log\log{q})^6}{(\log{q})^6}.
\]
This upper bound is stronger than the required estimate~\eqref{not too many characters}, which completes the proof of the theorem.
\end{proof}

\begin{rmk} It is known that $f(q) > q^{1/2}$ for almost all $q$, in the sense of asymptotic density. (This lower bound, and a bit more, follows from \cite[Theorem 1]{KP05}.) For such $q$, we deduce quickly from Lemma \ref{lem:rewrite} that the average of $n_{\chi}$ is $\ell(q) + O(q^{-1/50})$.  Thus, the estimate of Theorem \ref{thm:main} is interesting only for integers $q$ for which $f(q)$ is abnormally small. Our proof of Theorem \ref{thm:main} can be modified to show that the average of $n_{\chi}$ is always $\ell(q) + O\big( X/f(q) + \exp(-(\log\log{q})^{2}) \big)$ in the above notation, and that the term $X/f(q)$ can be omitted when $\ell (q) > X$. 
\end{rmk}

\section{The least character nonresidue averaged over the modulus~$q$}

Having established an asymptotic formula for the average of the least character nonresidue $n_\chi$ over all Dirichlet characters $\chi$ to a single modulus, we turn now to averaging $n_\chi$ over all characters with conductor less than~$x$. Theorem~\ref{thm:main} tells us that almost all of the $\phi(q)$ characters modulo $q$ have $n_\chi = \ell(q)$; therefore we should expect the average value of $n_\chi$ over all characters of conductor less than $x$ to be closely related to the sum $\sum_{q\le x} \phi(q)\ell(q)$. This is indeed the case, as the proof of Corollary~\ref{cor:finalavg} later in this section will show; the majority of the work is in showing how the constant $\Delta$ defined in equation~\eqref{eq:Delta def} arises in connection with that sum (see Proposition~\ref{prop:main sum}). At the end of this section, we outline how the same methods can be applied to averages over only certain characters; as a concrete example, we establish analogues of Theorem~\ref{thm:main} and Corollary~\ref{cor:finalavg} where we average only over primitive characters.

Our first two lemmas establish an asymptotic formula for the summatory function of $\phi(n)$ over integers relatively prime to a fixed number $m$, with enough uniformity in $m$ for our later purposes.

\begin{lem}\label{lem:Rankin twice}
Given a positive integer $m$, let $h(d)$ denote the largest divisor of $d$ that is coprime to~$m$. For all $x\ge3$,
\[
x^2 \sum_{d > x} \frac{\mu^2(d)}{dh(d)} + x\sum_{d\le x} \frac{\mu^2(d)}{h(d)} \ll 2^{\omega(m)} x\log x.
\]
\end{lem}

\begin{proof}
We employ Rankin's method. For any real number $0<\epsilon<1$, we have
\begin{multline*}
x^2 \sum_{d > x} \frac{\mu^2(d)}{dh(d)} + x\sum_{d\le x} \frac{\mu^2(d)}{h(d)} 
\\< x^2 \sum_{d > x} \frac{\mu^2(d)}{dh(d)} \bigg( \frac dx \bigg)^{1-\epsilon} + x \sum_{d\le x} \frac{\mu^2(d)}{h(d)} \bigg( \frac xd \bigg)^\epsilon = x^{1+\epsilon} \sum_{d=1}^\infty \frac{\mu^2(d)}{d^\epsilon h(d)}.
\end{multline*}
The right-hand side can be expressed as the convergent Euler product
\begin{align*}
x^{1+\epsilon} \prod_p \bigg( 1 + \frac1{p^\epsilon h(p)} \bigg) &= x^{1+\epsilon} \prod_{p\mid m} \bigg( 1 + \frac1{p^\epsilon} \bigg) \prod_{p\nmid m} \bigg( 1 + \frac1{p^{1+\epsilon}} \bigg) \\
&\le x^{1+\epsilon}  2^{\omega(m)} \prod_p \bigg( 1 + \frac1{p^{1+\epsilon}} \bigg) \\
&= x^{1+\epsilon}  2^{\omega(m)} \frac{\zeta(1+\epsilon)}{\zeta(2+2\epsilon)} \ll 2^{\omega(m)} x^{1+\epsilon} \epsilon^{-1}.
\end{align*}
Choosing $\epsilon = 1/\log x$ establishes the required estimate.
\end{proof}

\begin{lem}\label{lem:sumphi} Let $m$ be a natural number. For $x \ge 2$,
	\[ \sum_{\substack{n \leq x \\\gcd(n,m)=1}} \phi(n) = \frac{3x^2}{\pi^2} \prod_{p \mid m}\bigg(1+\frac1p\bigg)^{-1} + O(2^{\omega(m)} x\log x)
	\]
uniformly in $m$.
\end{lem}
\begin{proof} Let $\chi_0$ be the principal character modulo~$m$, and let $h(d)$ denote the largest divisor of $d$ that is coprime to~$m$ (as in Lemma~\ref{lem:Rankin twice}). The identity $\phi(n) \chi_0(n)/n = \sum_{d\mid n} \mu(d)/h(d)$ is easy to verify (since both sides are multiplicative functions, it suffices to check the identity on prime powers); it follows that
\begin{align*}
\sum_{\substack{n \leq x \\\gcd(n,m)=1}} \phi(n) &= \sum_{n\le x} n \cdot \frac{\phi(n) \chi_0(n)}n = \sum_{n \leq x}n\sum_{d\mid n} \frac{\mu(d)}{h(d)} \\
&= \sum_{d \leq x}\frac{\mu(d)}{h(d)}\sum_{\substack{n \le x \\ d\mid n}} n = \sum_{d \leq x}\frac{\mu(d)}{h(d)} \sum_{e \le x/d} de \\
&= \sum_{d \leq x} \frac{\mu(d)}{h(d)}\bigg(\frac{x^2}{2d} + O(x)\bigg) = \frac{x^2}{2}\sum_{d \leq x}\frac{\mu(d)}{dh(d)} + O\bigg(x\sum_{d\le x} \frac{\mu^2(d)}{h(d)}\bigg) \\
&= \frac{x^2}{2} \sum_{d=1}^\infty \frac{\mu(d)}{dh(d)} + O\bigg(x^2 \sum_{d > x} \frac{\mu^2(d)}{dh(d)} + x\sum_{d\le x} \frac{\mu^2(d)}{h(d)}\bigg).
\end{align*}
The error term is $\ll 2^{\omega(m)}x\log x$ by Lemma~\ref{lem:Rankin twice}, while in the main term, we have
\begin{align*}
\sum_{d=1}^\infty \frac{\mu(d)}{dh(d)} &= \prod_{p \nmid m}\bigg(1 - \frac{1}{p^2}\bigg) \prod_{p \mid m}\bigg(1-\frac{1}{p}\bigg)
\\ &= \prod_{p}\bigg(1-\frac{1}{p^2}\bigg) \prod_{p \mid m}\bigg(1+\frac{1}{p}\bigg)^{-1} = \frac{6}{\pi^2}\prod_{p \mid m}\bigg(1+\frac{1}{p}\bigg)^{-1}
	\end{align*}
as claimed.
\end{proof}

Our next goal is to evaluate the sum $\sum_{q\le x} \phi(q)\ell(q)$ that features prominently in the proof of Corollary~\ref{cor:finalavg}. Our method hinges upon sorting these integers $q$ by their largest divisor divisible only by primes not exceeding $2\log x$; the next two lemmas establish some preliminary results concerning sums over these friable numbers.

\begin{lem}\label{lem:tenenbaum}
Let $x\ge3$, and set $y=2\log x$.
\begin{enumerate}
\renewcommand{\labelenumi}{\rm(\alph{enumi})}
\item $\displaystyle \sum_{\substack{m > x^{1/2} \\ \text{$m$ is $y$-friable}}} \frac1m \ll x^{-1/3}$;
\item There are $\ll x^{2/3}$ integers less than $x$ that have a $y$-friable divisor exceeding $x^{1/2}$.
\end{enumerate}
\end{lem}

\begin{proof}
We may assume that $x$ is sufficiently large. It follows from work of de Bruijn that $\log \Psi(t,y) \ll (\log t)/\log y$ uniformly for $t\ge2$ and $y\le 4\log t$ (see for example \cite[Theorem 2, p.\ 359]{tenenbaum95}). In
particular, with $y=2\log{x}$, we have $\Psi(t,y) < t^{1/3}$ uniformly for $t\geq x^{1/2}$, once $x$ is sufficiently large. Hence,
\begin{align*}
\sum_{\substack{m > x^{1/2} \\ \text{$m$ is $y$-friable}}} \frac1m = \int_{x^{1/2}}^\infty \frac1t\, d\Psi(t,y)
&= -\frac{\Psi(x^{1/2},y)}{x^{1/2}} + \int_{x^{1/2}}^\infty \frac{\Psi(t,y)}{t^2}\,dt \\
&< 0 + \int_{x^{1/2}}^\infty \frac{t^{1/3}}{t^2}\,dt \ll x^{-1/3}
\end{align*}
when $x$ is sufficiently large, establishing part (a) of the lemma. Part (b) follows from part (a) upon noting that the number of integers less than $x$ that are divisible by a fixed $y$-friable integer $m>x^{1/2}$ is at most $x/m$ and summing over~$m$.
\end{proof}

\begin{lem}\label{lem:friable sum}
Let $x\ge3$, and set $y=2\log x$. Then
\[
\prod_{p \leq y}\bigg(1+\frac{1}{p}\bigg)^{-1} \sum_{\substack{m\le\sqrt x \\ \text{$m$ is $y$-friable}}} \frac{\phi(m)\ell(m)}{m^2} = \Delta + O\bigg( \frac1{\log x} \bigg),
\]
where $\Delta$ is defined in equation~\eqref{eq:Delta def}.
\end{lem}

\begin{proof}
Throughout we may assume that $x$ and hence $y$ are sufficiently large. We first consider the terms in the sum for which the value of $\ell(m)$ is fixed:
\begin{equation}\label{eq:for a single ell}
\sum_{\substack{m\le\sqrt x \\ \text{$m$ is $y$-friable} \\ \ell(m)=\ell}} \frac{\phi(m)\ell(m)}{m^2} = \ell \bigg( \sum_{\substack{m\ge1 \\ \text{$m$ is $y$-friable} \\ \ell(m)=\ell}} \frac{\phi(m)}{m^2} + O\bigg( \sum_{\substack{m>\sqrt x \\ \text{$m$ is $y$-friable} \\ \ell(m)=\ell}} \frac1m \bigg) \bigg).
\end{equation}
The sum inside the error term of equation~\eqref{eq:for a single ell}, even without the condition $\ell(m)=\ell$, is bounded by a constant times $x^{-1/3}$ by Lemma~\ref{lem:tenenbaum}(a). On the other hand, the first sum on the right-hand side of~\eqref{eq:for a single ell} can be written as the product
\begin{align*}
\sum_{\substack{m\ge1 \\ \text{$m$ is $y$-friable} \\ \ell(m)=\ell}} \frac{\phi(m)}{m^2} &= \prod_{p < \ell}\bigg(\frac{\phi(p)}{p^2} + \frac{\phi(p^2)}{p^4} + \dots\bigg) \prod_{\ell < p\leq y}\bigg(1+\frac{\phi(p)}{p^2} + \frac{\phi(p^2)}{p^4} + \cdots\bigg) \\
&=\bigg(\prod_{p < \ell}\frac{1}{p}\bigg)\prod_{\ell < p \leq y}\bigg(1+\frac{1}{p}\bigg).
\end{align*}
Using this information and summing equation~\eqref{eq:for a single ell} over the possible values of $\ell(m)$, we conclude that
\begin{align*}
\prod_{p \leq y}\bigg(1+\frac{1}{p}\bigg)^{-1} & \sum_{\substack{m\le\sqrt x \\ \text{$m$ is $y$-friable}}} \frac{\phi(m)\ell(m)}{m^2} \notag \\
&= \prod_{p \leq y}\bigg(1+\frac{1}{p}\bigg)^{-1} \sum_{\ell \leq y} \ell \bigg( \bigg(\prod_{p < \ell}\frac{1}{p}\bigg)\prod_{\ell < p \leq y}\bigg(1+\frac{1}{p}\bigg) + O(x^{-1/3}) \bigg) \notag \\
&= \sum_{\ell\le y} \frac{\ell^2}{\prod_{p\le \ell}(p+1)} + O(x^{-1/3} y^2) \notag \\
&= \Delta + O\bigg( x^{-1/4} + \sum_{\ell> y} \frac{\ell^2}{\prod_{p\le \ell}(p+1)} \bigg).
\end{align*}
It suffices to show that the sum in this error term is $O(1/\log x)$.

When $x$ is sufficiently large, each summand in the error term is less than half the previous summand. (Specifically, this is the case when $y>7$; here we use Bertrand's postulate to bound the ratio between consecutive primes by~$2$.) Consequently, that sum is bounded above by a geometric series that sums to twice the first term: if $r$ is the smallest prime exceeding $y$ (so that $r\le2y$), then
\[
\sum_{\ell> y} \frac{\ell^2}{\prod_{p\le \ell}(p+1)} < 2 \frac{r^2}{\prod_{p\le r}(p+1)} < 8y^2 \exp\bigg( {-} \sum_{p\le y} \log p \bigg) \ll y^2 e^{-c_5y}
\]
for some positive constant $c_5$, by Chebyshev's lower bound $\sum_{p\le y} \log p \gg y$. Since $y^3 e^{-c_5y}$ is a bounded function for $y\ge1$, we see that $y^2 e^{-c_5y} \ll 1/y \ll 1/\log x$, which completes the proof of the lemma.
\end{proof}

We now collect the previous three lemmas together to assist in our evaluation of the main sum appearing in Corollary~\ref{cor:finalavg}.

\begin{prop}\label{prop:main sum}
Let $\Delta$ be defined as in equation~\eqref{eq:Delta def}. Then for $x\ge3$,
\[
\sum_{q\le x} \phi(q)\ell(q) = \frac{3\Delta x^2}{\pi^2} +O\bigg( \frac{x^2}{\log x} \bigg)
\]
\end{prop}

\begin{proof}
As usual we may assume that $x$ is sufficiently large. We evaluate the sum by sorting the integers $q$ according to their divisors consisting only of small primes. Set $y=2\log{x}$, and define $Q=\prod_{p \leq y}p$. Let $M(q)$ denote the largest $y$-friable divisor of $q$; notice that $\ell(q)=\ell(M(q))$, since $\ell(q)<y$ for all $q\le x$ (when $x$ is sufficiently large).

For a fixed $y$-friable number $m$, the statement $M(q)=m$ is equivalent to $q$ being of the form $q'm$ with $\gcd(q',Q)=1$; note that in this case $\gcd(q',m)=1$ as well, and so $\phi(q)=\phi(q')\phi(m)$. Consequently, by Lemma~\ref{lem:sumphi},
\begin{align*}
\sum_{\substack{q \leq x\\ M(q)=m}} \phi(q) \ell(q) &= \phi(m)\ell(m) \sum_{\substack{q' \le x/m \\ \gcd(q', Q)=1}}\phi(q')\\
&= \phi(m)\ell(m) \frac{3(x/m)^2}{\pi^2} \prod_{p \leq y}\bigg(1+\frac{1}{p}\bigg)^{-1} + O\bigg(\phi(m)\ell(m) 2^{\pi(y)} \frac xm \log \frac xm \bigg).
\end{align*}
In the error term, note that $\phi(m)/m \le 1$ and $\ell(m)\log(x/m) \ll \log^2x$, and also that $2^{\pi(y)}\log^2x \ll \exp\big( O((\log x)/\log\log x) \big) \ll x^{1/3}$, say. We thus have
\[
\sum_{\substack{q \leq x\\ M(q)=m}} \phi(q) \ell(q) = \frac{3x^2}{\pi^2} \frac{\phi(m)\ell(m)}{m^2} \prod_{p \leq y}\bigg(1+\frac{1}{p}\bigg)^{-1} + O( x^{4/3} ).
\]

Summing over $m$, we conclude that
\begin{align*}
\sum_{q\le x} \phi(q)\ell(q) &= \sum_{\substack{m\le x \\ \text{$m$ is $y$-friable}}} \sum_{\substack{q \leq x\\ M(q)=m}} \phi(q) \ell(q) \\
&= \sum_{\substack{m\le\sqrt x \\ \text{$m$ is $y$-friable}}} \bigg( \frac{3x^2}{\pi^2} \frac{\phi(m)\ell(m)}{m^2} \prod_{p \leq y}\bigg(1+\frac{1}{p}\bigg)^{-1} + O( x^{4/3} ) \bigg) \\
&\qquad{}+ O\bigg( \sum_{\substack{\sqrt x < m \le x \\ \text{$m$ is $y$-friable}}} \phi(m)\ell(m) \bigg).
\end{align*}
In this last error term, each summand is at most $xy \ll x^{7/6}$, and the number of summands is $\ll x^{2/3}$ by Lemma~\ref{lem:tenenbaum}(b). Therefore this equation becomes
\[
\sum_{q\le x} \phi(q)\ell(q) = \frac{3x^2}{\pi^2} \prod_{p \leq y}\bigg(1+\frac{1}{p}\bigg)^{-1} \sum_{\substack{m\le\sqrt x \\ \text{$m$ is $y$-friable}}} \frac{\phi(m)\ell(m)}{m^2} + O(x^{11/6}),
\]
which by Lemma~\ref{lem:friable sum} becomes $\sum_{q\le x} \phi(q)\ell(q) = {3\Delta x^2/\pi^2} +O\big( {x^2/\log x} \big)$ as claimed.
\end{proof}

\begin{proof}[Proof of Corollary~\ref{cor:finalavg}]
We prove the corollary in the quantitative form
\[
\bigg( \sum_{q \leq x}\sum_{\substack{\chi \pmod{q}\\\chi\neq\chi_0}}1 \bigg)^{-1} \bigg( \sum_{q \leq x}\sum_{\substack{\chi \pmod{q}\\\chi\neq\chi_0}} n_{\chi} \bigg) = \Delta + O\bigg( \frac{(\log\log x)^2}{\log x} \bigg).
\]
Since the first double sum is well-known to equal
\[ \sum_{q \leq x}\sum_{\substack{\chi \pmod{q}\\\chi\neq\chi_0}}1 = \sum_{q \leq x}(\phi(q)-1) = \frac{3x^2}{\pi^2}  + O(x\log{x})\]
(the latter equality follows from Lemma \ref{lem:sumphi} with $m=1$), it suffices to show that
\begin{equation}\label{eq:numerator}
\sum_{q \leq x}\sum_{\substack{\chi \pmod{q}\\\chi\neq\chi_0}} n_{\chi} = \frac{3\Delta x^2}{\pi^2} + O\bigg( \frac{x^2(\log\log x)^2}{\log x} \bigg).
\end{equation}

By Theorem \ref{thm:main},
\begin{align*}
\sum_{q \leq x}\sum_{\substack{\chi \pmod{q}\\\chi\neq\chi_0}} n_{\chi} &= \sum_{3\le q \leq x}(\phi(q)-1)\bigg(\ell(q) + O\bigg(\frac{(\log\log{q})^2}{\log q}\bigg)\bigg) \\
&= \sum_{3\le q \leq x}\phi(q) \ell(q) + O\bigg(\sum_{3\le q \leq x}\bigg(\ell(q) +\phi(q) \frac{(\log\log{q})^2}{\log{q}}\bigg)\bigg).
\end{align*}
We note that $\ell(q) \ll \log x$ uniformly for $q \leq x$. Since the function $(\log\log t)^2/\log t$ is bounded for $t\ge3$ and decreasing for $t>e^{e^2}$, it follows that the error term (for $x$ sufficiently large) is
\[
\ll \sum_{q\le x} \log x + \sum_{q\le \sqrt x} q + \sum_{\sqrt x<q\le x} q \frac{(\log\log\sqrt x)^2}{\log\sqrt x},
\]
and so
\[ \sum_{q \leq x}\sum_{\substack{\chi \pmod{q}\\\chi\neq\chi_0}} n_{\chi}= \sum_{q \leq x}\phi(q) \ell(q) + O\bigg( \frac{x^2(\log\log x)^2}{\log x} \bigg). \]
Since $\sum_{q\le x} \phi(q)\ell(q) = {3\Delta x^2/\pi^2} +O\big( {x^2/\log x} \big)$ by Proposition~\ref{prop:main sum}, this equation establishes the formula~\eqref{eq:numerator} and hence the corollary.
\end{proof}

It is worth remarking that our methods can address the average value of $n_\chi$, not just over all nonprincipal characters $\chi$, but also over only certain characters. Let $q \ge 3$, and let $\X(q)$ be a nonempty collection of nonprincipal Dirichlet characters modulo~$q$. The set of $\chi \in \X(q)$ with $n_{\chi} > \ell(q)$ is obviously a subset of the set of all nonprincipal $\chi \pmod q$ with $n_{\chi}>\ell(q)$. This triviality, taken together with the estimates occurring in the proof of Theorem \ref{thm:main} and earlier in Section~\ref{main thm section}, shows that
\begin{equation}\label{eq:moregeneral}
\big( \#\X(q) \big)^{-1} \sum_{\chi \in \X(q)} n_{\chi} = \ell (q) + O\bigg(\frac{\phi(q)}{\#\X(q)}\frac{(\log\log q)^2}{\log q} \bigg).
\end{equation}

To take an example of special interest, let $\X(q)$ be the set of primitive characters modulo $q$, and let $\phi^*(q)=\#\X(q)$. From M\"{o}bius inversion applied to the relation $\sum_{d \mid q}\phi^*(d)=\phi(q)$, we find that
\[ \phi^*(q) = q \prod_{p \parallel q}\bigg(1-\frac{2}{p}\bigg) \prod_{p^2 \mid q}\bigg(1-\frac{1}{p}\bigg)^2. \]
Hence, $\phi^*(q) > 0$ precisely when $q \not\equiv 2\pmod{4}$, and whenever $\phi^*(q)$ is nonvanishing, we have
\[ \phi^*(q) \gg \phi(q) \prod_{p \mid q}\bigg(1-\frac{1}{p}\bigg) \gg \frac{\phi(q)}{\log\log q}. \]
So when $q \not\equiv 2\pmod{4}$, the estimate \eqref{eq:moregeneral} shows that the average of $n_{\chi}$ taken over primitive characters $\chi$ modulo $q$ is
\begin{equation}
\label{ell avg for primitive}
\frac1{\phi^*(q)} \sum_{\substack{\chi\pmod q \\ \chi\text{ primitive}}} n_\chi = \ell(q) + O\bigg( \frac{(\log\log q)^3}{\log q} \bigg),
\end{equation}
which is the analogue of Theorem~\ref{thm:main} for primitive characters. From this, we can deduce a corollary similar to Corollary~\ref{cor:finalavg}. It is necessary to replace Lemma~\ref{lem:sumphi} with the estimate
\begin{equation}\label{altsum}
\sum_{\substack{n \leq x \\\gcd(n,m)=1}} \phi^*(n) = \frac{18x^2}{\pi^4} \bigg(\prod_{p \mid m}\frac{p^3}{(p+1)(p^2-1)}\bigg) + O(2^{\omega(m)} x\log^2 x)
\end{equation}
uniformly in $m$, which can be proved by an argument similar to the proof of Lemma~\ref{lem:sumphi}. Imitating the proof of Corollary \ref{cor:finalavg} but using equations~\eqref{ell avg for primitive} and~\eqref{altsum} as input, we can establish:

\begin{thm}\label{finalavg primitive} Define
\[
\Delta^*= \sum_{\ell} \frac{\ell^4}{(\ell+1)^2(\ell-1)}\prod_{p < \ell}\frac{p^2-p-1}{(p+1)^2(p-1)} \approx 2.1514351057,
\]
where the sum and product are taken over primes $\ell$ and~$p$. Then
\begin{equation*}
\lim_{x\to\infty} \bigg( \sum_{q \leq x} \sum_{\substack{\chi\pmod q \\ \chi\text{ primitive}}} 1 \bigg)^{-1} \bigg( \sum_{q \leq x} \sum_{\substack{\chi\pmod q \\ \chi\text{ primitive}}} n_{\chi} \bigg) = \Delta^*.
\end{equation*}
\end{thm}

\section{The least non-split prime averaged over cubic number fields}

In this section, we consider cubic extensions $K$ of $\Q$ and investigate the distribution of the least rational prime with a particular factorization 
into prime ideas in the ring of integers $O_K$. This investigation is most successful in the case of the least prime that does not split completely, where we will be able to establish Theorem~\ref{thm:cubic}. In other cases (the least completely split prime, the least partially split prime, and the least inert prime), we will be able to establish the analogous Theorem~\ref{thm:generalsplit}, but only under the assumption of a generalized Riemann hypothesis.

The first ingredient of the proof of Theorem \ref{thm:cubic} is a now classical theorem of Davenport and Heilbronn \cite{DH71}:

\begin{prop}\label{prop:DH} As $x\to\infty$, the number of cubic fields $K$ with $|D_K| \leq x$ is $\sim \displaystyle\frac{x}{3\zeta(3)}$.\end{prop}

A refined version of the Davenport--Heilbronn theorem was proposed by Roberts \cite{roberts01} and recently confirmed by Taniguchi and Thorne~\cite{TT11} (see also the independent work of Bhargava, Shankar, and Tsimerman~\cite{BST10}). One consequence of their work (see \cite[Theorem 1.3 and Section 6.2]{TT11}) is the following explicit estimate:

\begin{prop}\label{prop:TT} Let $x \geq 1$, and let $1\le y \le \frac{1}{40}\log{x}$. For each prime $p \leq y$, we define local factors $t(p)$ and $t'(p)$, depending on the desired factorization of $p$, as follows:
\[ t(p) =
\begin{cases}
1/6 &\text{if $p$ is to split completely}, \\
1/2 &\text{if $p$ is to partially split}, \\
1/3 &\text{if $p$ is to be inert}, \\
1/p &\text{if $p$ is to partially ramify},\\
1/p^2 &\text{if $p$ is to ramify completely},
\end{cases}
\]
and
\[ t'(p) = \frac{t_p}{1+1/p + 1/p^2}. \]
The number of cubic number fields $K$ (up to isomorphism) with $|D_K| \leq x$ in which the primes $p\le y$ factor in the specified ways is
\[ \frac{x}{3\zeta(3)} \prod_{p \leq y}t'(p) + O(x^{5/6}), \]
uniformly in the choice of splitting conditions.
\end{prop}

If $K$ is a non-Galois cubic field,
then the normal closure of $K/\Q$ contains a unique quadratic subfield, called the \emph{quadratic resolvent} of $K$. If $K$ is cyclic, we adopt the convention that $K$ has quadratic resolvent $\Q$; in both cases, the quadratic resolvent is $\Q(\sqrt{D_K})$. The next lemma provides an upper bound on the number of cubic fields with a given quadratic resolvent. It should be noted that Cohen and Morra \cite{CM11} have asymptotic results for this problem for a \emph{fixed} quadratic resolvent, but in our application we require some uniformity.

\begin{lem}\label{lem:EV} Let $L$ be either $\Q$ or a quadratic extension of $\Q$. For $x\ge1$, the number of cubic fields whose discriminant is at most $x$ in absolute value, and whose quadratic resolvent equals $L$, is $\ll x^{0.84}$ uniformly in $L$.
\end{lem}

\begin{proof} Suppose that $K$ has quadratic resolvent $L$ and that $|D_K| \leq x$. Then $D_K = f^2 D_L$ for some positive integer $f$ (see for example \cite[\S6.4.5]{cohen93}).
Clearly $f \leq x^{1/2}$, and thus the number of choices for $D_K$ is also at most $x^{1/2}$. Ellenberg and Venkatesh \cite{EV07} have shown that the number of cubic fields with discriminant $D$ is $\ll_{\epsilon} |D|^{1/3+\epsilon}$, and so the lemma follows upon taking $\epsilon = \frac1{150}$ and summing over the $x^{1/2}$ possibilities for $D=D_K$.
\end{proof}

\begin{rmk}
The proof actually gives an upper bound of $\ll_{\epsilon} x^{5/6+\epsilon}/|D_L|^{1/2}$ for the number of such cubic fields, although we will not need this stronger statement.
\end{rmk}

Recall now that $n_K$ denote the least rational prime that does not split completely in~$K$; we are interested in the average value of $n_K$ as $K$ ranges over all cubic fields, ordered by the absolute value of their conductor. To help us handle the contribution to the average from fields $K$ for which $n_K$ is large, we quote two results from the literature. The first, a universal upper bound on $n_K$, is due to Li \cite{li11}:

\begin{prop}\label{prop:li} If $K$ is a cubic field, then the least non-split prime in $K$ is $\ll |D_K|^{1/7.39}$.
\end{prop}

\noindent As discussed by Li, it is much simpler to prove Proposition \ref{prop:li} with the larger exponent ${1}/{4\sqrt{\e}} + \epsilon$, which would also suffice for our purposes.

The next result, in slightly stronger form, appears without proof in a paper of Duke and Kowalski~\cite{DK00} (for details, see the proof of~\cite[Lemma 5.3]{pollack11}). Baier~\cite{baier06} has a sharper result allowing one to replace $2/A$ below with $1/(A-1)$, but we will not need this improvement.

\begin{prop}\label{prop:charlem} Fix $A> 2$ and $\epsilon>0$. For any $x\ge1$, the number of primitive Dirichlet characters $\chi$ of conductor at most $x$ with the property that $\chi(p)=1$ for all primes $p \leq (\log{x})^A$ is $\ll_{A,\epsilon} x^{2/A + \epsilon}$.
\end{prop}

\noindent This proposition quickly implies a lemma about the scarcity of {\em quadratic} fields with no small non-split primes.

\begin{lem}\label{lem:few quadratic fields}
The number of quadratic fields whose discriminant is at most $x$ in absolute value, in which every prime $p \le (\log x)^{200}$ splits, is $\ll x^{1/99}$.
\end{lem}

\begin{proof} There is a bijection between quadratic fields and primitive quadratic characters given by matching the quadratic field of discriminant $D$ with the primitive character $\leg{D}{\cdot}$ of conductor $|D|$ (see for example \cite[Theorem 9.13, p.\ 297]{MV07}). If $L$ is a quad\-rat\-ic field with the properties in the statement of the lemma, then $\chi(\cdot)=\leg{D_L}{\cdot}$ is a primitive character of conductor $|D_L| \le x$
with $\chi(p)=1$ for all $p \le (\log x)^{200}$. By Proposition~\ref{prop:charlem} with $\epsilon = \frac1{9900}$, the number of such characters is $\ll x^{1/99}$. 
\end{proof}

These preliminary field-counting results enable us to bound an error term that arises in the proof of Theorem~\ref{thm:cubic}.

\begin{lem}\label{lem:error terms field count}
Let $x\ge3$, and let $y=\frac1{40}\log x$ and $z=(\log x)^{200}$. There exists an absolute positive constant~$c_6$ such that
\[
z \sum_{\substack{|D_K| \le x \\ y<n_K}} 1 + x^{1/7.39} \sum_{\substack{|D_K| \le x \\ z<n_K}} 1 \ll x \exp\bigg( {-}\frac{c_6\log{x}}{\log\log{x}}\bigg),
\]
where the sums range over cubic fields $K$.
\end{lem}

\begin{proof}
The number of cubic fields $K$ with $|D_K|\leq x$ in which every prime $p \le y$ splits completely can be estimated by Proposition~\ref{prop:TT}, where we take $t_p = \frac16$ for every $p\le y$. In particular, each such $t_p'<\frac16$, and so the first term can be bounded by
\begin{align}
z \sum_{\substack{|D_K| \le x \\ y<n_K}} 1 \ll z \bigg( \frac x{6^{\pi(y)}} + x^{5/6} \bigg) &\ll (\log x)^{200} \bigg( \frac x{\exp(\log 6 \cdot \frac12y/\log y)} + x^{5/6} \bigg) \notag \\
&\ll x \exp\bigg( {-}\frac{c_6\log{x}}{\log\log{x}}\bigg)  \label{eq:6 to the pi(y)}
\end{align}
for some positive constant ($c_6=0.02$ is valid); we used in the middle estimate a crude Chebyshev-type lower bound for $\pi(y)$.

On the other hand, if $n_K>z$, then either $K$ is already a Galois extension (in which case its quadratic resolvent equals $\Q$) or else every prime $p \le z$ splits completely in the normal closure of $K$, and so also splits in the quadratic resolvent $L$ of $K$ (see \cite[Corollary, p. 106]{marcus77}). Lemma~\ref{lem:few quadratic fields} tells us that the number of such fields $L$ is $\ll x^{1/99}$; for each such $L$, the number of cubic fields $K$ whose quadratic resolvent equals $L$  is $\ll x^{0.84}$ by Lemma \ref{lem:EV}. Therefore the second term can be bounded by
\[
x^{1/7.39} \sum_{\substack{|D_K| \le x \\ z<n_K}} 1 \ll x^{1/7.39} x^{1/99 + 0.84} \ll x^{0.99},
\]
which is small enough to establish the desired estimate.
\end{proof}

With Proposition~\ref{prop:TT} still available to handle the smaller values of $n_K$, we are now ready to evaluate the average value of $n_K$ over cubic fields~$K$.

\begin{proof}[Proof of Theorem~\ref{thm:cubic}]
We prove the theorem in the quantitative form
\[
\bigg( \sum_{|D_K| \leq x} 1 \bigg)^{-1} \bigg( \sum_{|D_K| \leq x} n_K \bigg) = \Delta_{nsc} + O\bigg(\exp\bigg( {-}\frac{c_6\log{x}}{\log\log{x}}\bigg) \bigg),
\]
where the sums are taken over cubic fields~$K$. Taking $y=1$ in Proposition~\ref{prop:TT}, we deduce a quantitative version of the Davenport--Heilbronn Theorem (Proposition~\ref{prop:DH}), namely that the first sum is $x/3\zeta(3) + O(x^{5/6})$. Therefore it suffices to show that
\begin{equation}\label{eq:cubic denominator}
\sum_{|D_K| \leq x} n_K = \frac{x}{3\zeta(3)} \Delta_{nsc} + O\bigg( x \exp\bigg( {-}\frac{c_6\log{x}}{\log\log{x}}\bigg) \bigg).
\end{equation}

Set $y=\frac{1}{40}\log{x}$ and $z=(\log{x})^{200}$, and write
\begin{align}
\sum_{|D_K| \leq x} n_K &= \sum_{\substack{|D_K| \le x \\ n_K \le y}} n_K + \sum_{\substack{|D_K| \le x \\ y<n_K \le z}} n_K + \sum_{\substack{|D_K| \le x \\ z<n_K}} n_K \label{eq:three way prime split} \\
&= \sum_{\ell\le y} \ell \sum_{\substack{|D_K| \le x \\ n_K = \ell}} 1 + O\bigg( z \sum_{\substack{|D_K| \le x \\ y<n_K}} 1 + x^{1/7.39} \sum_{\substack{|D_K| \le x \\ z<n_K}} 1 \bigg); \notag
\end{align}
in the last sum we have used the universal upper bound for $n_K$ from Proposition~\ref{prop:li}. By Lemma~\ref{lem:error terms field count}, this estimate becomes
\begin{equation}\label{eq:nly main term left}
\sum_{|D_K| \leq x} n_K = \sum_{\ell\le y} \ell \sum_{\substack{|D_K| \le x \\ n_K = \ell}} 1 + O\bigg( x \exp\bigg( {-}\frac{c_6\log{x}}{\log\log{x}}\bigg) \bigg).
\end{equation}

For any prime $\ell$, the cubic fields $K$ such that $n_K=\ell$ are exactly those fields in which $p$ splits completely for all $p<\ell$, yet $\ell$ does not split completely. Proposition~\ref{prop:TT}, summed over the four other possibilities for the factorization of $\ell$, tells us that the number of such fields with $|D_K| \leq x$ is
\[
\frac x{\zeta(3)} \frac{(5/6 + 1/\ell+1/\ell^2)}{1+1/\ell+1/\ell^2} \prod_{p < \ell}\frac{1/6}{1+1/p+1/p^2} + O(x^{5/6}),
\]
and therefore
\begin{align*}
\sum_{\ell\le y} \ell \sum_{\substack{|D_K| \le x \\ n_K = \ell}} 1 &= \frac{x}{3\zeta(3)} \sum_{\ell \leq y} \frac{\ell(5/6 + 1/\ell+1/\ell^2)}{1+1/\ell+1/\ell^2} \prod_{p < \ell}\frac{1/6}{1+1/p+1/p^2} + O(y^2 x^{5/6}) \\
&= \frac{x}{3\zeta(3)} \Delta_{nsc} + O\bigg( x \sum_{\ell > y} \Big( \ell \prod_{p < \ell} \tfrac16 \Big) + x^{5/6} \log^2 x\bigg)
\end{align*}
by the definition~\eqref{eq:Delta_nsc def} of $\Delta_{nsc}$. In this remaining sum, each summand is at most $\frac13$ of the previous one by Bertrand's postulate, and so the sum is bounded (up to a multiplicative constant) by the first term; we conclude that
\begin{align*}
\sum_{\ell\le y} \ell \sum_{\substack{|D_K| \le x \\ n_K = \ell}} 1 &= \frac{x}{3\zeta(3)} \Delta_{nsc} + O( x \cdot 2y \cdot 6^{-\pi(y)} + x^{5/6} \log x ) \\
&= \frac{x}{3\zeta(3)} \Delta_{nsc} + O\bigg( x \exp\bigg( {-}\frac{c_6\log{x}}{\log\log{x}}\bigg) \bigg)
\end{align*}
just as in the estimate~\eqref{eq:6 to the pi(y)}. Inserting this result into equation~\eqref{eq:nly main term left} establishes equation~\eqref{eq:cubic denominator} and hence the theorem.
\end{proof}

For each natural number $k$ and each prime $p \equiv 1\pmod{k}$, let $r_k(p)$ denote the least prime $k$th power residue modulo $p$. Elliott \cite{elliott71} has shown  that for each of $k=2, 3$, and $4$, the function $r_k(p)$ possesses a finite mean value. When $k=3$, Elliott's result gives the average smallest split prime in cubic extensions of prime conductor.

Motivated by Elliott's work, one might wonder if it is possible to obtain the average least split prime, where the average is instead taken over {\em all} cubic extensions of $\Q$ (ordered by the absolute value of their discriminant as in Theorem~\ref{thm:cubic}). One could ask the same question for the least partially split prime or the least inert prime. We can establish the following conditional results, where the averages are taken over cubic fields~$K$:

\begin{thm}\label{thm:generalsplit} Assume the generalized Riemann hypothesis for Dedekind zeta functions. For each prime $p$, define
\[ t'_{\text{cs}}(p) = \frac{1/6}{1+1/p+1/p^2}, \quad t'_{\text{inert}}(p) = \frac{1/3}{1+1/p+1/p^2}, \quad\text{and } t'_{\text{ps}}(p) = \frac{1/2}{1+1/p+1/p^2}. \]
The average least completely split prime in a cubic field is
\[ \sum_{\ell} \ell t'_{\text{cs}}(\ell) \prod_{p < \ell}(1-t'_{\text{cs}}(p)) = 19.7952216366\ldots. \]
The average least inert prime is
\[ \sum_{\ell} \ell t'_{\text{inert}}(\ell) \prod_{p < \ell}(1-t'_{\text{inert}}(p)) = 8.5447294614\ldots. \]
Finally, if the average is restricted to non-Galois cubic fields, the average least partially split prime is
\[ \sum_{\ell} \ell t'_{\text{ps}}(\ell) \prod_{p < \ell}(1-t'_{\text{ps}}(p)) = 5.3680248421\ldots. \]
Here all sums and products are taken over primes $\ell$ and~$p$.
\end{thm}

The proofs mimic that of Theorem \ref{thm:cubic}. However, the proofs (in particular the analogues of Lemma~\ref{lem:error terms field count}) are simpler in that the parameter $z$ is no longer required: all values of $n_K$ greater than $y$ in equation~\eqref{eq:three way prime split} can be treated together. The reason is that under the generalized Riemann hypothesis, the least unramified prime $p$ with a prescribed splitting type is $\ll \log^2{|D_K|}$ (see \cite{LO77}, \cite{LMO79}); this conditional universal upper bound for the analogues of $n_K$ is small enough to obviate the need for the second splitting at~$z$.

It would be interesting to find an unconditional proof for any of the assertions of Theorem~\ref{thm:generalsplit}. The obstacle at present is the lack of an unconditional universal upper bound for the analogues of $n_K$ that could take the place of Proposition~\ref{prop:li}. Currently the best known upper bound for the least completely split prime is $|D|^{c_7}$ for a 
certain unspecified (and potentially large) absolute constant $c_7$ 
(see \cite[Theorem 1.1]{LMO79}), 
and similarly for the least inert prime or the least partially split prime.

We conclude with a remark that might be illuminating. Each of the constants in Theorem~\ref{thm:generalsplit} has the form $\sum_\ell \ell t(\ell) \prod_{p<\ell} \big(1-t(\ell)\big)$, where $t(\ell)$ is the ``probability'' that the desired property holds for the prime~$\ell$. This sum is exactly the form the expectation would have if each prime were replaced by a random event $X_\ell$ that took place with probability $t(\ell)$, as long as the events $X_\ell$ were independent. The constant $\Delta_{nsc}$ in Theorem~\ref{thm:cubic} also has this form, and in fact so do the constants on the right-hand sides of equations~\eqref{eq:erdos1} and even~\eqref{eq:erdos0} (where $t(p)=\frac12$ reflects the fact that each prime is expected to be a quadratic residue exactly half of the time). The calculations of these average-value results demonstrate that these properties of small primes are indeed asymptotically independent; in the aforementioned situations it is quite helpful that the sums converge so rapidly that only the small primes are relevant.

\section*{Acknowledgements}
The authors thank Carl Pomerance for helpful conversations and Mits Kobayashi for technical assistance with Karim Belabas's \texttt{cubics} package.

\bibliographystyle{amsalpha}
\bibliography{ALCNFVTE}

\end{document}